\documentclass[11pt]{article}

\usepackage{amsmath, amsfonts, amsthm, graphicx,amssymb}
\usepackage[top=21mm, bottom=22mm, left=21mm, right=21mm]{geometry}
\usepackage{comment}
\usepackage{hyperref}
\hypersetup{
	colorlinks=true,
	linkcolor=blue,
	filecolor=magenta,      
	urlcolor=cyan,
	citecolor=blue
}
\usepackage{authblk}
\usepackage{enumitem}
\usepackage{aliascnt} 
\usepackage[none]{hyphenat}
\usepackage{tikz, tikz-3dplot}
\usetikzlibrary{calc, patterns,intersections, scopes}
\usetikzlibrary{positioning,arrows,shapes,decorations.markings,decorations.pathreplacing, decorations.pathmorphing,matrix,patterns}
\tikzstyle{vtx}=[circle,draw=black,fill=black,inner sep=0,minimum size=5pt,text=white,font=\footnotesize]

\newtheorem{theorem}{Theorem}[section]

\newaliascnt{lemma}{theorem}
\newtheorem{lemma}[lemma]{Lemma}
\aliascntresetthe{lemma}
\newaliascnt{claim}{theorem}
\newtheorem{claim}[claim]{Claim}
\aliascntresetthe{claim}
\newaliascnt{question}{theorem}

\aliascntresetthe{question}

\usepackage{cleveref}
\crefname{lemma}{lemma}{lemmas}
\Crefname{lemma}{Lemma}{Lemmas}
\crefname{claim}{claim}{claims}
\Crefname{claim}{Claim}{Claims}
\crefname{question}{question}{questions}
\Crefname{question}{Question}{Questions}

\theoremstyle{definition}

\newcommand{\cF}{\mathcal{F}}

\title{Cross-free families have linear size}
\author{Istv\'an Tomon}
\affil{Ume\r{a} University}
\affil{\texttt{istvan.tomon@umu.se}}
\date{}

\begin{document}

\maketitle
\sloppy

\begin{abstract}
    Two subsets $A$ and $B$ of a ground set $X$ are \emph{crossing} if all four regions of their Venn diagram are non-empty; a family $\cF\subset 2^X$ is \emph{$k$-cross-free}, if it contains no $k$ pairwise crossing members. Such families arise in multiflow problems, combinatorial geometry, and phylogenetic combinatorics. Nearly fifty years ago, Karzanov and Lomonosov conjectured that every $k$-cross-free family on an $n$-element ground set has size $O(kn)$. We prove the bound $O_k(n)$, settling (arguably) the main problem about the growth rate of such families. 
    
    In fact, we prove the stronger statement that every \emph{$k$-wise laminar} family has at most linear size. Here a family is \(k\)-wise laminar if it contains no antichain of size \(k\) whose members have nonempty common intersection.


\end{abstract}

\noindent
\textbf{Update.} After the publication of this manuscript on arXiv, we were informed  that the main result of the paper (\Cref{thm:main1}) was already proved by Knill \cite{Knill} in 1994. From this, linear bound for the Karzanov-Lomonosov conjecture (\Cref{thm:main0}) follows immediately. This paper was completely overlooked, and for the past three decades, the conjecture was continuously stated to be open \cite{DKM,Fleiner,GKMW,KPT}. The author is grateful to Hongyi Gou, Mostafa Mirabi, Mihir Mittal, Chieu-Minh Tran, and Zhenyu Yang for finding the paper of Knill.

Our paper provides a substantially different proof of Knill's result. We will not publish  the current manuscript in a journal.

\section{Introduction}

Let $X$ be an $n$-element ground set. Two sets $A,B\subset X$ are \emph{crossing} if none of the four sets $A\setminus B, B\setminus A, A\cap B,$ and $X\setminus(A\cup B)$ are empty. Equivalently, $A$ and $B$ are crossing if their Venn diagram has no empty regions. In this paper, we address a long-standing conjecture of Karzanov and Lomonosov \cite{KL} regarding the maximum size of families containing no $k$ pairwise crossing members, referred to as \emph{$k$-cross-free} families.

Note that the definition of crossing is invariant under complementation of its arguments; if $\cF\subset 2^X$ is $k$-cross-free and $A\in \cF$,  one may include $X\setminus A$  without violating the $k$-cross-free property.  A 2-cross-free family is equivalent to a \emph{laminar} family (after accounting for complements), where a family is laminar if any two members are either disjoint, or one contains the other. A classical result of Edmonds and Giles \cite{EG} states that the maximum size of a laminar family is $2n$, implying the bound of $4n-2$ for 2-cross-free families.

The importance of 3-cross-free families stems from the locking theorem of Karzanov and Lomonosov \cite{KL} (1978), which characterizes the families of subsets $\mathcal{F} \subset X$ that can be simultaneously "locked" by a multiflow --- meaning the flow value across each cut $(A, X \setminus A)$ for $A \in \mathcal{F}$ achieves the cut's minimum capacity. They proved that a family is lockable in every graph if and only if it is 3-cross-free, providing a powerful multiflow generalization of the Max-flow min-cut theorem. This motivated the search for a general bound for the size of $k$-cross-free families. Karzanov and Lomonosov \cite{Karzanov,KL}, and Pevzner \cite{Pevzner} conjectured that the maximum size of a $k$-cross-free family on an $n$-element ground set is $O(kn)$. Early progress by Lomonosov provided an elegant $O(kn\log n)$ upper bound, derived from the observation that any element $x \in X$ can be contained in at most $k-1$ sets of a fixed size $\ell < n/2$.  For the case $k=3$, Pevzner \cite{Pevzner} proved a linear bound, and a very short argument for the bound $10n$ was given by Fleiner \cite{Fleiner}. This was subsequently sharpened to the optimal $8n-20$ by Dress, Koolen, and Moulton \cite{DKM}.

The existence of a linear bound for $k \geq 4$ remained open. Gr\"unewald, Koolen, Moulton and Wu \cite{GKMW} proved linearity for specific subclasses of 4-cross-free families, but the general $O(kn \log n)$ bound persisted as the state-of-the-art for decades. In 2019, Kupavskii, Pach, and Tomon \cite{KPT} achieved the first general improvement $O_k(n\log^* n)$, coming tantalizingly close to linearity. In this work, we close the gap.

\begin{theorem}\label{thm:main0}
    Every $k$-cross-free family $\cF\subset 2^{[n]}$ has size at most $O_k(n)$.
\end{theorem}

In their survey of research problems in discrete geometry, Brass, Moser, and Pach \cite{BMP} highlight interesting connections between the Karzanov-Lomonosov conjecture and geometric crossing problems. A fundamental question in the theory of graph drawings is whether there exists a constant $c_k$ such that every \emph{geometric graph} (i.e. graph drawn in the plane with straight line edges) on $n$ vertices with at least $c_k n$ edges contains $k$ pairwise crossing edges, where crossing is meant in the usual geometric sense.

By Euler’s polyhedral formula, a geometric graph with no $2$ crossing edges has at most $3n-6$ edges. The existence of $c_3$ was proved by Agarwal, Aronov, Pach, Pollack, and Sharir \cite{AAPPS} (see also \cite{AT,PRT}), and the existence of $c_4$ was confirmed by Ackerman \cite{Ackerman}. However, for $k\geq 5$, the problem remains open with the best known general bound being $O_k(n \log n)$ due to Valtr \cite{Valtr}.

Although "crossing" in the geometric sense differs from the set-theoretic definition, the two concepts coincide in case the vertices are in convex position. Capoyleas and Pach \cite{CP} proved that a geometric graph on $n$ vertices in convex position with no $k$ pairwise crossing edges has at most $2(k-1)n - \binom{2k-1}{2}$ edges, and this bound is optimal for $n\geq 2k$. This result is equivalent to stating that if $X$ is a cyclically ordered set and $\cF \subset 2^X$ is a $k$-cross-free family of intervals, then $|\cF| \leq 4(k-1)n - 2\binom{2k-1}{2}$. Our main result extends this linearity from the restricted case of intervals to arbitrary families of subsets.

Beyond geometry and optimization, $k$-cross-free families play a structural role in phylogenetic combinatorics, particularly in the study of split systems and circular networks \cite{DHKMS}. 

\medskip

In fact, we prove a strengthening of \Cref{thm:main0}, which might be of independent interest. Say that a family $\mathcal{F}\subset 2^X$ is \emph{$k$-wise laminar}, if for any $k$ distinct sets $A_1,\dots,A_k\in \mathcal{F}$, if $A_1\cap\dots\cap A_k$ is non-empty, then there exist $1\leq i\neq j\leq k$ such that $A_i\subset A_j$. In other words, $\cF$ does not contain an antichain of size $k$ whose elements have a non-empty common intersection. We prove the following theorem, and show that it indeed implies \Cref{thm:main0}.

\begin{theorem}\label{thm:main1}
    Every $k$-wise laminar family $\cF\subset 2^{[n]}$ has size at most $O_k(n)$.
\end{theorem}

\noindent
The remainder of the paper is devoted to the proofs of \Cref{thm:main0} and \Cref{thm:main1}.

\section{Preliminaries}

In this section, we collect a few simple results and introduce our notation. We start by recalling some basic terminology about set systems. Let $\cF\subset 2^X$.
\begin{itemize}
    \item  Two sets $A,B\subset X$ are \emph{comparable} if $A\subset B$ or $B\subset A$, otherwise incomparable.
    \item $\cF$ is a \emph{chain} if any two of its members are comparable. Equivalently, $\cF$ is a chain if there is an enumeration $A_1\subset \dots\subset A_s$ of the elements of $\cF$. Also, $\cF$ is a \emph{continuous chain} if $|A_{i+1}|=|A_i|+1$ for $i=1,\dots,s-1$.
    \item $\cF$ is an  \emph{antichain} if any two elements of $\cF$ are incomparable.
    \item $\cF$ is \emph{fully-intersecting} if $\bigcap_{A\in \cF} A$ is non-empty. Also, $\cF$ is \emph{$k$-wise intersecting} if the intersection of any $k$ elements of $\cF$ is non-empty.
\end{itemize}
Using this terminology, a family is $k$-wise laminar if it does not contain a fully-intersecting antichain of size~$k$. The next simple lemma shows that \Cref{thm:main1} indeed implies \Cref{thm:main0}.

\begin{lemma}\label{lemma:wkcf}
Let $\cF\subset 2^X$ be a $k$-cross-free family. Then there exists a $k$-wise laminar family $\cF'\subset 2^X$ of size at least $|\cF|/2$.
\end{lemma}

\begin{proof}
Say that $A$ and $B$ are \emph{weakly-crossing} if the three sets $A\setminus B$, $B\setminus A$ and $A\cap B$ are all non-empty. Say that a family is \emph{weakly-$k$-cross-free} if it does not contain $k$ pairwise weakly-crossing sets. Observe that every weakly-$k$-cross-free family is also $k$-wise laminar.

Let $x\in X$ be arbitrary and let $\cF_0=\{A\in \cF:x\not\in A\}$. Note that the union of no two sets in $\cF_0$ is $X$, so $\cF_0$ is weakly-$k$-cross-free. If $|\cF_0|\geq |\cF|/2$,  then $\cF'=\cF_0$ suffices. Otherwise, let $\cF'=\{X\setminus A:A\in \cF\setminus\cF_0\}$. Then $\cF'$ is also $k$-cross-free and no element of $\cF'$ contains $x$, so it is also weakly-$k$-cross-free, and has size at least $|\cF|/2$.
\end{proof}

Dilworth's theorem \cite{Dilworth} states that  any partially ordered set (such as a set-system ordered by the subset relation) can be partitioned into as many chains as the size of the largest antichain. This has the following simple corollary about $k$-wise laminar families. 

\begin{lemma}\label{lemma:dilworth}
Let $\cF\subset 2^X$ be a $k$-wise laminar and $k$-wise intersecting family. Then $\mathcal{F}$ can be partitioned into $k-1$ chains.
\end{lemma}

\begin{proof}
 $\cF$ does not contain an antichain of size $k$, as the $k$-wise intersecting property would imply that it is a fully-intersecting antichain of size $k$. But then Dilworth's theorem \cite{Dilworth} implies that $\cF$ is the union of $k-1$ chains.
\end{proof}

\begin{lemma}\label{lemma:uniform}
Let $\ell\geq 1$, and let $\cF\subset 2^X$ be a $k$-wise laminar family such that $|A|=\ell$ for every $A\in \cF$. Then 
$$|\cF|\leq \frac{(k-1)|X|}{\ell}.$$ 
\end{lemma}

\begin{proof}
By simple averaging, some $x\in X$ is contained in at least $|\cF|\ell/|X|$ elements of $\cF$. As no two $\ell$-element sets are comparable, we must have $|\cF|\ell/|X|\leq k-1$.
\end{proof}

Finally, we make repeated use of the following formulation of Tur\'an's theorem \cite{turan} from extremal graph theory.

\begin{lemma}\label{lemma:turan}
Every $n$ vertex graph of average degree $d$ contains an independent set of size at least $\frac{n}{d+1}$.
\end{lemma}

\section{\texorpdfstring{$k$}{k}-wise laminar families}

In this section, we prove \Cref{thm:main1}, which then immediately implies \Cref{thm:main0} as well following \Cref{lemma:wkcf}.

\begin{theorem}\label{thm:main}
    For every $k$, there exists $c_k>0$ such that the following holds. Let $X$ be a set of size $n$, and let $\cF\subset 2^X$ be a $k$-wise laminar family. Then $|\cF|\leq c_kn$.
\end{theorem}

We prove \Cref{thm:main} by induction on $k$. In case $k=2$, $\cF$ is a laminar family, and we can prove that $c_2=2$ suffices by an elementary inductive argument. In the rest of the proof, we assume that $k\geq 3$ and that $c_{k-1}$ exists. The proof consists of two main parts. First, we show that if $\cF$ is extremal with respect to being $k$-wise laminar, and $\cF$ has $cn$ elements for some sufficiently large constant $c$, then we can find a dense collection of long, continuous chains. Second, we build a tree encoding certain special comparabilities between the chains. We argue that if we are able to build a sufficiently large such tree, then we can extract $k$ sets violating the $k$-wise laminar condition, reaching a contradiction.

\medskip

Say that $\cF\subset 2^X$ is \emph{extremal}, if $\cF$ is $k$-wise laminar, and $\cF$ maximizes the quantity $|\cF'|/|X'|$ among all $k$-wise laminar families $\cF'\subset 2^{X'}$ with $|X'|\leq |X|$. Clearly, in order to prove \Cref{thm:main}, it is enough to only consider extremal families. Let $\cF\subset 2^X$ be an extremal family, and assume that $|\cF|\geq cn$ for some sufficiently large constant $c>0$ specified later. Our goal is to arrive at a contradiction. First, we show that $\cF$ contains a dense collection of long continuous chains.

\begin{lemma}\label{lemma:chain_extraction}
Let $h$ and $\alpha$ be positive integers. If $c$ is sufficiently large with respect to $h$ and $\alpha$, then $\cF$ contains a collection $\{C_i\}_{i\in I}$ of pairwise disjoint continuous chains such that $|I|\geq \alpha n$ and $|C_i|=h+1$ for $i\in I$.
\end{lemma}

\begin{proof}
Let $x\in X$, and consider the family $\cF_x=\{A\setminus \{x\}:A\in \cF\}\subset 2^{X\setminus\{x\}}$. Observe that $\cF_x$ is also $k$-wise laminar. Therefore, as $\cF$ is extremal, we have
$$\frac{|\cF|}{n}\geq \frac{|\cF_x|}{n-1},$$
which implies $|\cF_x|\leq |\cF|-|\cF|/n$. Therefore, there are at least $|\cF|/n$ sets $A\in \cF$ such that $x\not\in A$ and $A\cup\{x\}\in \cF$ as well.

Define the auxiliary directed graph $H$ on vertex set $\cF$, where there is an edge pointing from $A$ to $B$ if $B=A\cup\{x\}$ for some $x\in X\setminus A$. The previous argument implies that $H$ has at least $n\cdot (|\cF|/n)=|\cF|$ edges.

Say that a set $A\in \cF$ is \emph{exceptional} if there exist $x,y\in X\setminus A$, $x\neq y$, such that $A\cup\{x\},A\cup\{y\}\in \cF$. Equivalently, $A$ is exceptional if the outdegree of $A$ in $H$ is at least 2. Let $\cF'$ be the family of exceptional sets. 
\begin{claim}
    $|\cF'|\leq kc_{k-1}n+k^2n$.
\end{claim}

\begin{proof}
Remove all members of $\cF'$ of size at most $k$, and let $\cF'_1$ be the resulting family. By \Cref{lemma:uniform}, we have 
$$|\cF'_1|\geq |\cF'|-1-\sum_{\ell=1}^k\frac{(k-1)n}{\ell}\geq |\cF'|-k^2n.$$ Let $Q$ be the graph on vertex set $\cF'_1$ in which $A,B\in \cF'_1$ are joined by an edge if $|A|=|B|$ and $|A\Delta B|=2$. Observe that if $B_1,\dots,B_k$ are neighbors of $A$ in $Q$, then $B_1\cap\dots\cap B_k\neq \emptyset$, so $\{B_1,\dots,B_k\}$ is a fully-intersecting antichain violating the $k$-wise laminar condition. Therefore, the maximum degree of $Q$ is at most $k-1$. Let $\cF'_2\subset \cF'_1$ be an independent set of $Q$ of size $|\cF'_2|\geq |\cF'_1|/k$, whose existence is guaranteed by Tur\'an's theorem (\Cref{lemma:turan}).

The main observation is that $\cF'_2$ is $(k-1)$-wise laminar. Indeed, assume to the contrary that $\{A_1,\dots,A_{k-1}\}\subset \cF'_2$ is a fully-intersecting antichain of size $k-1$. Assume that $|A_{k-1}|$ is minimal among $|A_1|,\dots,|A_{k-1}|$.  As $A_{k-1}$ is exceptional, there exist $x,y\in X\setminus A_{k-1}$, $x\neq y$, such that $A_{k-1}\cup\{x\},A_{k-1}\cup\{y\}\in \cF$. But then $$\{A_1,\dots,A_{k-2},A_{k-1}\cup\{x\},A_{k-1}\cup\{y\}\}$$ is a fully-intersecting antichain of size $k$, which  is fairly straightforward to check. Indeed, the intersection of these $k$ sets is non-empty, as it contains the intersection of $A_1,\dots,A_{k-1}.$  Moreover, there is no comparability between these sets. In order to see this, it is enough to show that $A_{k-1}\cup\{x\}$ and $A_i$ are incomparable for $i\in [k-2]$: we cannot have $A_{k-1}\cup\{x\}\subset A_i$ as $A_{k-1}\not\subset A_i$; and we cannot have $A_{i}\subset A_{k-1}\cup\{x\}$ as either $|A_i|>|A_{k-1}|$, or $|A_i|=|A_{k-1}|$ and $|A_i\Delta A_{k-1}|\geq 4$. 

In conclusion, $\cF'_2$ is $(k-1)$-wise laminar, so $|\cF'_2|\leq c_{k-1}n$. But then $$|\cF'|\leq |\cF'_1|+k^2n\leq k|\cF'_2|+k^2n\leq kc_{k-1}n+k^2n.$$
\end{proof}

Let $\cF_1=\cF\setminus\cF'$ and let $H_1$ be the subgraph of $H$ spanned by $\cF_1$. Observe that every vertex of $H$, with the possible exception of the empty set, has at most $2k$ neighbors. This is true as the outneighbors of any non-empty set form a fully-intersecting antichain, and any $k$ inneighbors of a set of size at least $k+1$ form a fully-intersecting antichain (while a set of size at most $k$ has at most $k$ inneighbors as well). Therefore, $H_1$ has at least 
$$e(H)-2k|\cF'|-n\geq |\cF|-dn$$
edges for some constant $d=d(k)$, where the $-n$ term accounts for the possible removal of the empty set. In $H_1$, the outdegree of every vertex is at most 1, and the indegree of every vertex is at most $k$. The former implies (after also noting that $H$ contains no directed cycles) that every connected component of $H_1$ is a directed rooted tree, whose edges are directed towards the root.

Let $\{C_i\}_{i\in I}$ be a maximum-size collection of pairwise disjoint continuous chains in $\cF_1$, each $C_i$ of size $h+1$. Each such chain $C_i$ is a directed path of length $h$ in $H_1$. Let $\cF_2=\cF_1\setminus\bigcup_{i\in I}C_i$, and let $H_2$ be the subgraph of $H_1$ induced by $\cF_2$. Then the number of edges of $H_2$ is at least 
$$e(H_2)\geq e(H_1)-\sum_{i\in I}(k+1)|C_i|\geq |\cF|-dn-(k+1)(h+1)|I|.$$
 By the maximality of $\{C_i\}_{i\in I}$, $H_2$ does not contain a directed path of size $h+1$. But each connected component of $H_2$ is a directed rooted tree, whose edges are directed towards the root, and each vertex has indegree at most $k$. If such a tree does not contain a directed path of size $h+1$, its size is at most $1+k+\dots+k^{h-1}=:K$. This implies that the total number of edges of $H_2$ is at most 
$$e(H_2)\leq \frac{K-1}{K}|\cF_2|\leq \frac{K-1}{K}|\cF|.$$ Comparing the lower and upper bounds on the number of edges of $H_2$, we get
$$|\cF|-dn-(k+1)(h+1)|I|\leq \frac{K-1}{K}|\cF|,$$
which implies
$$|I|\geq \frac{|\cF|-Kdn}{(k+1)(h+1)K}\geq \frac{c-Kd}{(k+1)(h+1)K}n.$$
Set $c:=\alpha(k+1)(h+1)K+Kd$. Then $c$ only depends on $k,h,\alpha$, we have $|I|\geq \alpha n$, and thus the collection $\{C_i\}_{i\in I}$ satisfies the required properties.
\end{proof}

Let $h=(2k)^{2k^2+1}$ and $\alpha=32(h+2)!k^{h+2}$, and let $c$ be the constant guaranteed by \Cref{lemma:chain_extraction}. Let $\{C_i\}_{i\in  I_0}$ be a collection of pairwise disjoint continuous chains in $\cF$, each of size $h+1$. For every $i\in I_0$, the members of $C_i$ can be written as 
$$A\subset A\cup\{x_1\}\subset A\cup\{x_1,x_2\}\subset\dots\subset A\cup\{x_1,\dots,x_h\}$$ with a suitable set $A$ and elements $x_1,\dots,x_h\in X\setminus A$. Define 
$$X_i:=\{x_1,\dots,x_h\}.$$ Moreover, for $x\in X_i$, let $C_i(x)$ denote the largest set in $C_i$ not containing $x$. Then $C_i(x),C_i(x)\cup\{x\}\in C_i$.

Next, we prove that we can select a dense subcollection of these chains with several additional properties. In particular, we want to find $I\subset I_0$ such that the following conditions are satisfied.
\begin{description}
    \item[C1] If $i,j\in I$ and $x\in X_i\cap X_j$, then $C_i(x)$ and $C_j(x)$ are comparable.
    \item[C2] There is a total ordering $\prec$ of $X$ such that if $i\in I$ and $x,y\in X_i$ with $x\prec y$, then $C_i(x)\subset C_i(y)$.
    \item[C3] If $i,j\in I$ and $X_i\cap X_j\neq \varnothing$, then either every member of $C_i$ has smaller size than every member of $C_j$, or the other way around.
    \item[C4] If $i\in I$, then every member of $C_i$ has size at least $3k^2h$.
\end{description}

\begin{lemma}
    There exists $I\subset I_0$ of size $|I|\geq (kh+1)n$ such that \textbf{C1}-\textbf{C4} are satisfied.
\end{lemma}

\begin{proof}
    For $x\in X$, let 
    $$S_x=\{C_i(x)\cup\{x\}:i\in I_0,x\in X_i\}.$$ Then $S_x$ is a $k$-wise laminar  and fully-intersecting family. By \Cref{lemma:dilworth}, $S_x$ can be partitioned into $k-1$ chains. Let $D_x$ be one of these chains, chosen randomly from the uniform distribution on the $k-1$ chains. Let $I_1$ be the set of indices $i\in I_0$ such that for every $x\in X_i$, the set $C_i(x)\cup\{x\}$ is in the selected chain $D_x$. Then each $i\in I_0$ is included in $I_1$ with probability $(k-1)^{-h}$, so $\mathbb{E} |I_1|=|I_0|(k-1)^{-h}$. Fix a choice of chains $D_x$ such that $|I_1|\geq \mathbb{E} |I_1|$, and observe that $I_1$ satisfies \textbf{C1}.

    Let $\prec$ be a random ordering of $X$, chosen from the uniform distribution on all $n!$ orderings. Let $I_2$ be the set of indices $i\in I_1$ such that for every $x,y\in X_i$, if $x\prec y$, then $C_i(x)\subset C_i(y)$. Then for every $i\in I_1$, we have $i\in I_2$ with probability $1/h!$. Therefore,  $\mathbb{E}|I_2|=|I_1|/h!$. Fix an ordering $\prec$ such that $|I_2|\geq \mathbb{E}|I_2|$. Then $I_2$ satisfies \textbf{C2}.

    Let $H$ be the auxiliary graph on $I_2$ in which $i$ and $j$ are joined by an edge if $X_i\cap X_j\neq \varnothing$, and some member of $C_i$ has the same size as some member of $C_j$. Then every $i\in I_2$ has degree at most $4h^2$. Indeed, for every $x\in X_i$, the sets $\{C_j(x):j\in I_2,x\in X_j\}$ form a chain, so there are at most $4h$ indices $j\in I_2$ with $x\in X_j$ and $||C_j(x)|-|C_i(x)||\leq 2h$. But if $||C_j(x)|-|C_i(x)||>2h$, then no member of $C_i$ can have the same size as a member of $C_j$. Let $I_3$ be a maximum-size independent set of $H$, then $I_3$ satisfies \textbf{C3} as well. By Tur\'an's theorem, $|I_3|\geq |I_2|/(4h^2+1)$.
    
    Finally, remove all indices $i\in I_3$ such that the minimum of $C_i$ has size less than $3k^2h$, and let $I$ be the resulting set. By \Cref{lemma:uniform}, we have $$|I|\geq |I_3|-1-\sum_{\ell=1}^{3k^2h}(k-1) n/\ell\geq |I_3|-3k^3hn.$$
    Therefore, $$|I|\geq \frac{|I_0|}{(4h^2+1)h!(k-1)^h}-3k^3hn\geq (kh+1)n.$$
 \end{proof}

In the rest of the proof of \Cref{thm:main}, we fix $I\subset I_0$ of size $|I|\geq (kh+1)n$ satisfying \textbf{C1}-\textbf{C4}. With the help of these chains, we define a tree encoding comparabilities between the chains, with a number of other advantageous properties.

We introduce some notation and terminology related to trees. Given a rooted tree $T$, we imagine $T$ being drawn in the plane without crossing edges such that the root is at the bottom, and vertices of the same \emph{depth} (i.e. distance from the root) lie on a horizontal line (see \Cref{fig:1}). Their order on this line bears significance. A vertex in $T$ is \emph{leftmost} if it is the leftmost vertex on its horizontal line.  Given a vertex $v\in V(T)$, we write $T[v]$ for the subtree spanned by $v$ and its descendants. Finally, $T$ is \emph{perfect} if every root-to-leaf path has the same length, which is called the \emph{height} of the tree.

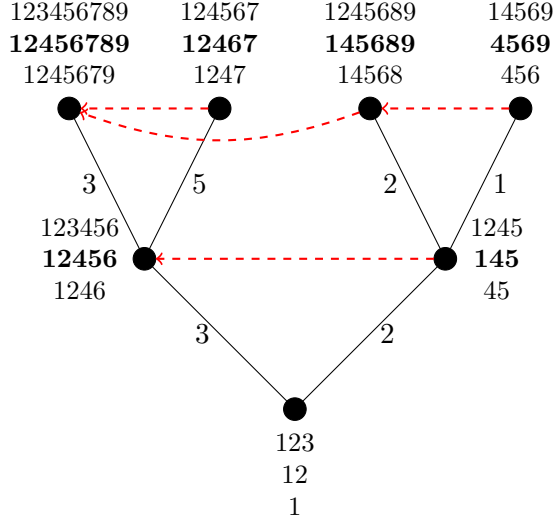
\begin{figure}
\begin{center}
\begin{tikzpicture}[
    % 1. Define the dot style (don't put this in 'nodes={}')
    dot/.style={
        circle,
        draw=black,
        fill=black,
        minimum size=3mm,
        inner sep=0pt,
        outer sep=0pt
    },
    % 2. Define how all labels and edge nodes should look (no circles!)
    every label/.style={
        align=center,
        inner sep=5pt,
        font=\small
    },
    every edge from parent node/.style={
        align=center,
        font=\footnotesize,
        inner sep=10pt
    },
    level distance=2cm,
    level 1/.style={sibling distance=4cm},
    level 2/.style={sibling distance=2cm},
    grow=up
]

% The Tree Structure
\node[dot, label=below:{123\\12\\1}] {}
    child {
        node[dot, label=right:{1245\\\textbf{145}\\45}] (a) {}
        child {node[dot,label=above:{14569\\\textbf{4569}\\456}] (b) {} edge from parent node[right] {1}}
        child {node[dot,label=above:{1245689\\\textbf{145689}\\14568}] (c) {} edge from parent node[left] {2}}
        % Edge label for the left branch
        edge from parent node[right] {2}
    }
    child {
        node[dot, label=left:{123456\\\textbf{12456}\\1246}] (e) {}
        child {node[dot,label=above:{124567\\\textbf{12467}\\1247}] (f) {} edge from parent node[right] {5}}
        child {node[dot,label=above:{123456789\\\textbf{12456789}\\1245679}] (g) {}edge from parent node[left] {3}}
        % Edge label for the right branch
        edge from parent node[left] {3}
    };
   \draw[red, dashed, thick, ->] (b) -- (c) ;
   \draw[red, dashed, thick, ->] (a) -- (e) ;
   \draw[red, dashed, thick, ->] (f) -- (g) ;
    \draw[red, dashed, thick, ->] (c) to[bend left=20] (g) ;
\end{tikzpicture}
\caption{An example of a cross-support tree. Each node represents a chain of size 3. We use the shorthand 123 for $\{1,2,3\}$. The sets $S_v$ are written in bold, and the red edges highlight the containment between the sets $S_v$ required by \textbf{T5}.}
\label{fig:1}
\end{center}
\end{figure}

The following object is the heart of the proof of \Cref{thm:main}. A \textbf{cross-support tree} is a perfect rooted tree $T$ with the following properties:
\begin{description}
    \item[T1] Every vertex $v\in V(T)$ is an element of $I$, and each edge $uv\in E(T)$ receives a label $\rho(uv)\in X$. (Formally, each vertex $v$ is labeled with an element of $I$, so we allow different vertices to have the same label. To simplify notation, we view $v$ as an element of $I$, which should not lead to any confusion.)
    \item[T2] If $v\in V(T)$, the label of every edge with endpoint $v$ is an element of $X_v$. If $u_1,\dots,u_d$ are the children of $v$ from left-to-right, then $\rho(vu_1)\succ\dots\succ\rho(vu_d)$.
    \item[T3] If $v$ is not a root or a leaf, the parent of $v$ is $p$ and its leftmost child is $u_1$, then $\rho(pv)=\rho(vu_1)$.  If $v$ is not a root and has parent $p$, define $$\phi(v):=\rho(pv)\quad\text{and}\quad S_v:=C_v(\phi(v)).$$
     \item[T4] Let $v\in V(T)$ be a non-leaf and let $u$ be a child of $v$. If $x=\rho(vu)$, then $C_{v}(x)\subsetneq C_{u}(x)$. 
    \item[T5] Let $u,u'\in V(T)$ be distinct vertices at the same depth such that $u$ is to the left of $u'$. Let $a$ be their common ancestor of maximum depth, and let $v\neq v'$ be the children of $a$, where $v$ is the ancestor of $u$ (or $v=u$), and $v'$ is the ancestor of $u'$ (or $v'=u'$). If $u$ is a leftmost vertex in $T[v]$, then $S_{u'}\subseteq S_u$.
\end{description}
See \Cref{fig:1} for an example of a cross-support tree. Let $T$ be a cross-support tree with root $\mathbf{r}$ and height at most $k$. Then we can deduce the following additional properties of $T$.

\begin{description}
    \item[T6] Let $u,v\in V(T)$ such that $v\neq \mathbf{r}$ and $u$ is a leftmost vertex in $T[v]$. Then $\phi(u)=\phi(v)$ and if $p$ is the parent of $v$, then $$C_p(\phi(v))\subsetneq S_v\subseteq S_u,$$
    with the second containment strict if $v\neq u$.
    \item[T7] The family $\bigcup_{u\in V(T)}C_{u}$ is $k$-wise intersecting.
    \item[T8] If $v\neq \mathbf{r}$ is an ancestor of $u$, then $|S_v|<|S_u|$.
    \item[T9] After removing some children of $\mathbf{r}$ together with all their descendants, the resulting tree is also a cross-support tree.
\end{description}

\begin{lemma}\label{lemma:T}
$T$ satisfies \textbf{T6}-\textbf{T9}. Moreover, \textbf{T6}-\textbf{T8} do not use \textbf{T5}.
\end{lemma}
\begin{proof}
\ 
\begin{description}
    \item[T6:] Let $v=v_0,\dots,v_q=u$ be a path from $v$ to $u$. Then $v_i$ is the leftmost child of $v_{i-1}$ for $i=1,\dots,q$, so $\phi(v_i)=\phi(v_{i-1})$ by \textbf{T3}. But then writing $x=\phi(v)=\phi(v_i)$, we also have $S_{v_i}=C_{v_i}(x)$, so $S_v=S_{v_0}\subsetneq\dots \subsetneq S_{v_q}=S_u$ follows from \textbf{T4}. Moreover, $\rho(pv)=\phi(v)=x$, so $C_p(x)\subset C_v(x)=S_v$ also follows from \textbf{T4}.
    \item[T7:] Let $B$ be the smallest member of $C_{\mathbf{r}}$. We prove by induction that if $u$ is at depth $d$ and $A\in C_u$, then $|A\cap B|\geq |B|-dh$. As $|B|\geq 3k^2h$ by \textbf{C4} and the height of $T$ is at most $k$, this implies that for any choice of $A_1,\dots,A_k\in \bigcup_{u\in V(T)} C_u$, we have
    $$|A_1\cap\dots\cap A_k\cap B|\geq |B|-k^2h>0,$$ which then implies \textbf{T7}. The statement holds for $d=0$, so assume that $d\geq 1$. Let $p$ be the parent of $u$, then $C_p(x)\subset C_u(x)$ for $x=\rho(pu)$ by \textbf{T4}. By our induction hypothesis, $|C_p(x)\cap B|\geq |B|-(d-1)h$, so $|C_u(x)\cap B|\geq |B|-(d-1)h$ as well. But $|C_u(x)\Delta A|\leq h$, so $|A\cap B|\geq |B|-dh$.
    \item[T8:] It is enough to prove the statement in case $v$ is the parent of $u$, as then the claim follows by induction. But this follows from \textbf{C3}, as $X_u\cap X_v$ contains $x=\rho(uv)$, and $C_v(x)\subsetneq C_u(x)$ by \textbf{T4}.
    \item[T9:] It is clear that the resulting tree $T'$ satisfies \textbf{T1}-\textbf{T4}. In order to see that \textbf{T5} is also true, note that for any $v\in V(T')$ with $v\neq \mathbf{r}$, we have $T[v]=T'[v]$.
\end{description}
\end{proof}

The following lemma is the first key result about cross-support trees, showing that they contain large fully-intersecting antichains.

\begin{lemma}\label{lemma:kcross}
If there exists a cross-support tree $T$ of height $k$ where every non-leaf vertex has at least $k+1$ children, then $\mathcal{F}$ contains a fully-intersecting antichain of size $k$.
\end{lemma}

\begin{proof}
Let $\mathbf{r}=v_0,\dots,v_k$ be a root-to-leaf path such that $v_i$ is not the leftmost child of $v_{i-1}$, and $\phi(v_i)$ is distinct from $\phi(v_1),\dots,\phi(v_{i-1})$. We can find such a path greedily by noting that for every $v\in V(T)$, the edges from $v$ to the children of $v$ receive different labels (\textbf{T2}). For $i=1,\dots,k$, let 
$$A_i:=S_{v_i}\cup\{\phi(v_i)\}\in C_{v_i}.$$ We prove that $\{A_1,\dots,A_k\}$ is a fully-intersecting antichain. First, we have
$$(A_1\cap \dots \cap A_k)\supset (S_{v_1}\cap\dots \cap S_{v_k})\neq \varnothing \quad\text{by \textbf{T7}},$$
so $\{A_1,\dots,A_k\}$ is fully-intersecting. If it is not an antichain, then there exist $1\leq i<j\leq k$ such that $A_i$ and $A_j$ are comparable. But
$$|A_i|=|S_{v_i}|+1<|S_{v_j}|+1=|A_j|\hspace{20pt}\mbox{by \textbf{T8}},$$
so this is only possible if $A_i\subset A_j$. We show that this is impossible. Let $w$ be the leftmost vertex in $T[v_i]$ on the same depth as $v_j$. Then $\phi(w)=\phi(v_i)$ by  \textbf{T6}. Let $x=\phi(w)=\phi(v_i)$, and recall that $x\neq \phi(v_j)$. As $v_{i+1}$ is not an ancestor of $w$ (as the ancestor of $w$ is the leftmost child of $v_i$), the common ancestor of $w$ and $v_j$ is $v_i$. Therefore, \textbf{T5} implies that $S_{v_j}\subset S_w$. Crucially, as $x\not\in S_w$, we have $x\not\in S_{v_j}$ and thus $x\not\in S_{v_j}\cup\{\phi(v_j)\}=A_j$. As $x\in A_i$, we cannot have $A_i\subset A_j$.
\end{proof}

 To finish the proof, we construct a cross-support tree of height $k$, whose every non-leaf vertex has at least $k+1$ children. We build such a tree by induction.

 \begin{lemma}
 For $\ell=0,1,\dots,k$, there exists $I_{\ell}\subset I$ of size $|I_{\ell}|\geq |I|-\ell h n$, and for every $i\in I_{\ell}$, there exists a cross-support tree $T_{\ell,i}$ such that 
\begin{itemize}
    \item[(1)] the height of $T_{\ell,i}$ is $\ell$,
    \item[(2)] the root of $T_{\ell,i}$ is $i$,
    \item[(3)] every non-leaf vertex of $T_{\ell,i}$ has at least $h/(2k)^{k\ell}$ children.
\end{itemize}
\end{lemma}
\begin{proof}
We proceed by induction on $\ell$. The case $\ell=0$ follows by taking $I_0=I$, and defining $T_{0,i}$ to be the single vertex tree, whose root is $i$. Now assume that $\ell>0$ and that the statement holds for $\ell-1$ instead of $\ell$.

For every $i\in I_{\ell-1}$, let $Y_i=X_i$ if $\ell=1$, and let $Y_i$ be the set of labels on the edges connecting the root $i$ of $T_{\ell-1,i}$ to its children if $\ell\geq 2$. Then $Y_i\subseteq X_i$ by \textbf{T2} and $|Y_i|\geq h/(2k)^{k(\ell-1)}$. Let $Z_i$ be the second half of $Y_i$ with respect to the ordering $\prec$ of $X$, so $Z_i\subset Y_i$ and $|Z_i|\geq |Y_i|/2$. For each $x\in X$, let 
$$J_x=\{i\in I_{\ell-1}:x\in Z_i\}.$$
By \textbf{C1}, the family $\{C_i(x):i\in J_x\}$ is a chain. Let $J^{\text{top}}_x\subset J_x$ be the set of indices corresponding to the $h$ largest members of this chain, or $J^{\text{top}}_x=J_x$ if $|J_x|<h$. Define 
$$I_{\ell}:=I_{\ell-1}\setminus \bigcup_{x\in X}J_x^{\text{top}},$$
then $|I_{\ell}|\geq |I_{\ell-1}|-hn\geq |I|-h\ell n$.

For every $i\in I_{\ell}$, we define the desired tree $T_{\ell,i}$ as follows. For every $x\in Z_i$, there are at least $h$ indices $j\in I_{\ell-1}$ such that $x\in Z_j$ and $C_i(x)\subset C_j(x)$. This is true because $i\in J_x\cap I_{\ell}$, which means that $|J_x^{\text{top}}|=h$ and every $j\in J_x^{\text{top}}$ works. Therefore, for each $x\in Z_i$, we can pick an index $j_x\in J_x^{\text{top}}$ such that $j_x\neq j_y$ for $x\neq y\in Z_i$. If $\ell=1$, set $T'_x=T_{0,j_x}$. If $\ell\geq 2$, then for each tree $T_{\ell-1,j_x}$, remove all children of the root $j_x$ (and their descendants), whose joining edge is labeled by an element of $X$ that is larger than $x$ in the ordering $\prec$, and let $T'_{x}$ be the resulting tree. By \textbf{T9}, $T'_x$ is also a cross-support tree. Recalling that $x\in Z_{j_x}$, $x$ is in the second half of $Y_{j_x}$ with respect to $\prec$, so the root of $T_{x}'$ has degree at least $h/(2(2k)^{k(\ell-1)})>h/(2k)^{k\ell}$. To summarize, we have the following properties of $T_x'$:
\begin{itemize}
    \item $T_{x}'$ is a cross-support tree with root $j_x$,
    \item if $\ell\geq 2$, the leftmost edge from the root is labeled with  $x$,
    \item $C_i(x)\subset C_{j_x}(x)$,
    \item every non-leaf vertex has at least $h/(2k)^{k\ell}$ children. 
\end{itemize}

Now define the tree $T'$ by taking the disjoint union of the trees $T_{x}'$ for $x\in Z_i$, and joining the root $j_x$ of $T_x'$ to the new root $i$. The trees $T_{x}'$, $x\in Z_i$, are drawn from left to right with respect to the reverse of the order $\prec$ on $X$.  We label the edge $ij_x$ by $x$. Clearly, the highlighted properties of the trees $T_x'$ ensure that $T'$ satisfies \textbf{T1}-\textbf{T4}, which then imply \textbf{T6}-\textbf{T8} by \Cref{lemma:T}. However, \textbf{T5} does not necessarily hold, so we remove some further children of the root of $T'$.

 For $d=0,\dots,\ell-1$, let $i_{d,x}$ be the leftmost vertex of $T'_x$ at depth $d$. The family 
$$\{S_{i_{d,x}}:x\in Z_i\}$$ is $k$-wise intersecting by \textbf{T7}. Therefore, applying \Cref{lemma:dilworth} in a total of $\ell$ times, there exists $Q\subset Z_i$ of size at least 
$$|Q|\geq \frac{|Z_i|}{(k-1)^{\ell}}\geq \frac{|Y_i|}{2(k-1)^{\ell}}\geq \frac{h/(2k)^{(\ell-1) k}}{2(k-1)^{\ell}}>\frac{h}{(2k)^{k\ell}}$$
such that $$D_d=\{S_{i_{d,x}}:x\in Q\}$$ is a chain for every $d=0,\dots,\ell-1$. Let $T=T_{\ell,i}$ be the tree we get by keeping only those subtrees $T'_x$ of $T'$ for which $x\in Q$. 

\begin{claim}
$T$ satisfies \textbf{T5}.
\end{claim}

\begin{proof}
First, we show that if $x,y\in Q$ and $y\prec x$, then $S_{i_{d,y}}\subset S_{i_{d,x}}$. As $D_d$ is a chain, we have that $S_{i_{d,x}}$ and $S_{i_{d,y}}$ are comparable. As $i_{d,x}$ is a leftmost vertex in $T_x'=T[j_x]$, we have $C_{i}(x)\subset S_{j_x}\subset S_{i_{d,x}}$ by \textbf{T6}. But $y\prec x$, so $C_{i}(y)\cup\{y\}\subseteq C_{i}(x)$ by \textbf{C2}, and in particular $y\in C_i(x)$. As $i_{d,y}$ is a leftmost vertex in $T_{y}'$, we have $\phi(i_{d,y})=y$ by \textbf{T6}, so $S_{i_{d,y}}=C_{i_{d,y}}(y)$. In particular, $y\not\in S_{i_{d,y}}$. Thus, as $S_{i_{d,x}}$ and $S_{i_{d,y}}$ are comparable, we must have $S_{i_{d,y}}\subset S_{i_{d,x}}$.

In order to prove that $T$ satisfies \textbf{T5}, we need to verify the following. Let $u,u'\in V(T)$ be distinct vertices at the same depth such that $u$ is to the left of $u'$. Let $a$ be their common ancestor of maximum depth, and let $v\neq v'$ be the children of $a$, where $v$ is the ancestor of $u$ (or $v=u$), and $v'$ is the ancestor of $u'$ (or $v'=u'$). If $u$ is a leftmost vertex in $T[v]$, then $S_{u'}\subseteq S_u$.

If the common ancestor $a$ is not the root $i$, then this is true using that $u$ and $u'$ lie in some cross-support subtree $T'_x$, $x\in Q$. Hence, assume that $a=i$, then $v=j_x$ and $v'=j_y$ for some $x,y\in Q$, $y\prec x$. Also, $u$ is a leftmost vertex in $T_x'$, so $u=i_{d,x}$ for some $d$. We have $S_{u'}\subset S_{i_{d,y}}$ using that $u',i_{d,y}$ are vertices of the cross-support tree $T_y'$. But then $S_{u'}\subset S_{i_{d,y}}\subset S_{i_{d,x}}=S_u$, verifying the required claim.
\end{proof}

In conclusion, $|I_{\ell}|\geq |I|-\ell h n$, and for every $i\in I_{\ell}$, the tree $T=T_{\ell,i}$ is a cross-support tree of height $\ell$ in which every non-leaf vertex has at least $h/(2k)^{k\ell}$ children. This finishes the proof.
\end{proof}

Setting $\ell=k$, and noting that $I_k$ is nonempty, there exists a cross-support tree $T$ of height $k$ in which every non-leaf vertex has at least $h/(2k)^{k^2}\geq k+1$ children. By \Cref{lemma:kcross}, $\cF$ contains a fully-intersecting antichain of size $k$, contradiction. Thus, we must have $|\cF|\leq cn$, finishing the proof of \Cref{thm:main}.

\section*{Acknowledgments}

IT acknowledges the support of the Swedish Research Council grant VR 2023-03375.

\end{document}